\documentclass[a4paper,10pt]{article}

\title{Two-generated verbally closed subgroups of a free solvable group $G$ are retracts of $G$}
\author{V.A. Roman'kov, E.I. Timoshenko}

\date{}

\usepackage{amsmath,amsthm,amsfonts}
\usepackage{graphicx}
\usepackage{hyperref}
\usepackage[usenames]{color}

\newtheorem{theorem}{Theorem}[section]
\newtheorem{lemma}[theorem]{Lemma}
\theoremstyle{definition}

\newtheorem{definition}[theorem]{Definition}
\newtheorem{example}[theorem]{Example}
\newtheorem{problem}[theorem]{Problem}

\newtheorem{proposition}[theorem]{Proposition}

\newcounter{comcount}

\begin{document}

\maketitle

\begin{abstract}
We prove that every verbally closed two-generated subgroup of a free  solvable  group $G$ of a finite rank is a retract  of $G.$  
\end{abstract}

%\tableofcontents

\section{Introduction}
\label{se:intro}

Algebraically closed objects play an important part in modern algebra.    In this paper we study verbally closed and algebraically closed subgroups of free solvable groups. 

We  first recall that a  subgroup $H$ is called {\it algebraically closed} in a group $G$ if  for every finite system of equations $S = \{E_i(x_1, \ldots,x_n, H) = 1 \mid i = 1, \ldots, m\}$ with constants from  $H$ the following holds: if $S$  has a solution in $G$ then it has a solution in $H$. Then we recall  that a  subgroup $H$ of a group $G$  is called a  {\em retract } of $G$, if there is a homomorphism (termed {\em  retraction}) $\phi: G \to H$ which is identical on $H$.
It is easy to show   that every retract of $G$ is algebraically closed in $G$. Furthermore, if   $G$ is finitely presented and $H$ is finitely generated then the converse is also true (see \cite{MR}).  This result still holds for any finitely generated group $G$ which is {\em equationally Noetherian}. Recall that a group $G$ is called {\em equationally Noetherian} if for any $n$ every system of equations in $n$ variables with coefficients from $G$ is equivalent (has the same solution set in $G$)  to  some finite subsystem of itself (see  \cite{BMRom}, \cite{BMR1}.

Let $F(X)$ be the free group of countably infinite  rank with basis $X=\{x_1, x_2, ...,  x_n, ... \}.$ For  $w= w(x_1, \ldots,x_n) \in F(X)$ and a group $G$ by $w[G]$ we denote the set of all $w$-elements in $G$, i.e., $w[G] = \{w(g_1, \ldots,g_n) \mid g_1, \ldots, g_n \in G\}$.  The {\em verbal subgroup} $w(G)$ is the subgroup of $G$ generated by $w[G]$.
The {\em $w$-width}  $l_{w}(g) = l_{w, G}(g)$  of an element $g \in w(G)$ is the  minimal natural number $n$ such that $g$ is a product of  $n$ $w$-elements in $G$ or their inverses; the width of $w(G)$ is the supremum of widths of its elements.    The first  question about $w$-width   goes back to the Ore's paper \cite{Ore} where he asked whether  the $[x,y]$-width (the {\em commutator} width) of every element in a non-abelian finite simple group is equal to $1$ ({\em Ore Conjecture}). The conjecture was established   in \cite{LBST}. Observing paper  \cite{Romeq} and monographs \cite{Segal},  \cite{Romess}  present results about verbal width in groups. 

Two important questions arise naturally for an extension $H \leq G$ and a  given word $w \in F(X)$:
\begin{itemize}
\item when it is true that   $w[H] = w[G] \cap H$?
\item when $l_{w, G}(h) = l_{w,H}(h)$ for a given $h \in w(H)$? 
\end{itemize} 

  To approach these questions Mysnikov and the author  introduced in \cite{MR} a new notion of verbally closed subgroups.
  \begin{definition}
\label{de:1.1}
  A subgroup $H$ of $G$ is called  {\em verbally closed} if for any word $w \in F(X)$ and element $h \in H$  equation $w(x_1, \ldots, x_n) = h$ has a solution in $G$ if and only if it has a solution in $H$, i.e., $w[H] = w[G] \cap H$ for every $w \in F(X)$.
  \end{definition} 

\medskip
\noindent
\begin{theorem}  \cite{MR}. 
\label{th:1.2}
Let $F$ be a  free group of a finite rank. Then for a subgroup $H$  of $F$ the following conditions are equivalent:
\begin{enumerate}
\item [a)] $H$ is a retract of $F$.
\item [b)] $H$ is a verbally closed  subgroup of $F$.
\item [c)] $H$ is an algebraically closed subgroup of $F$.
\end{enumerate}
\end{theorem}

\medskip
This result clarifies the nature of verbally or algebraically closed subgroups in $F$. Surprisingly, the "weak" verbal closure operator in this case is as strong as  the standard one.

Similar result is true for any free nilpotent group $N$ of a finite rank.
\begin{theorem} \cite{RH}.
\label{th:1.3}
Let  $N$ be a free nilpotent group of a finite rank and  class $c$. 
Then for a subgroup $H$  of $N$ the following conditions are equivalent:
\begin{enumerate}
\item [a)] $H$ is a retract of $N$.
\item [b)] $H$ is a verbally closed  subgroup of $N$.
\item [c)] $H$ is an algebraically closed subgroup of $N$.
\item [d)] $H$ is  a free factor of the free group $N$ in the variety ${\cal N}_c$ of all nilpotent groups of the class $\leq c.$
\end{enumerate}
\end{theorem}

\section{Preliminaries}
\label{se:prelim}

In this section we collect some known or simple facts on verbally or algebraically closed subgroups.

\begin{proposition} \cite{MR}.
\label{pr:2.1} {\it Let $H \leq G$ be a group  extension. Then the following holds:
\begin{itemize}
\item[1)] If $H$ is a retract of $G$ then $H$ is algebraically closed in $G.$
\item [2)]  Suppose $G$ is finitely presented and $H$ is finitely generated. Then $H$ is algebraically closed in $G$ if and only if  $H$ is a retract.
\item [3)] Suppose $G$  is finitely generated relative to $H$ and $H$ is 
equationally Noetherian. Then $H$ is algebraically closed in $G$ if and only if  $H$ is a retract.
\end{itemize}}
\end{proposition}

Further in the paper $M_r$ denotes a free metabelian group of rank $r$ and $M_{rk}$ denotes a free metabelian nilpotent of class $k$ group of rank $r.$ 

The following statement was proved in \cite{Romeqmet} (see also \cite{Romess}). 

\begin{lemma}
\label{le:2.2}
 Let $G$ denotes  $M_2$ or $M_{rk}, k \geq 4,$  with basis $z_1, z_2$. Let tuple $(g_1, g_2)$ is a solution of the equation 

 \begin{equation}
\label{eq:1}
              (x_2,x_1,x_1,x_2) = (z_2,z_1,z_1,z_2).                                                    
\end{equation}
Then  $ g_i \equiv  z_i^{\pm 1} (mod\ G'), i = 1, 2.$
\end{lemma}

\begin{lemma} [\cite{RH}, Lemma 1.1]
\label{le:2.3} Let $ G$ be a group  and let $N$ be a verbal subgroup of $G.$ If  $H$ is a verbally closed subgroup of $G$ then its  image $H_N$ is  verbally closed in  $G_N=G/N$. 
\end{lemma}
\begin{proof} Suppose that equation $w(x_1, ..., x_n) = h$ ($h \in H_N$) has a solution  $(g_1, ..., g_n)$ in $ G_N$. Let $(g_1', ..., g_n')$ be a preimage of this solution in $G.$   Then $(g_1',..., g_n')$ is a solution of the equation  $w(x_1, ..., x_n) = f'$ in $G.$ Here  $f' = w(g_1',..., g_n')$ is a preimage of $ h $ in $G.$ There also is a preimage $h'$  of $h$ in $H.$ Then $ f'= h'v, v \in  N$.  Let  $v= v(x_{n+1}, ..., x_{n+m})$ be a corresponding to $N$ word  that has value 1 in $G/N.$  The equation $w(x_1, ..., x_n)v(x_{n+1},..., x_{n+m})^{-1} = h'$  is  solvable in $G$. Then it has a solution in $ H.$ The image of the $n$ first components of this solution will be a solution of $w(x_1,..., x_n) = h$ in   $H_N.$
\end{proof}

 \begin{lemma}
\label{le:2.4}
Let $H = gp(g, f)$ be a two-generated verbally closed noncyclic subgroup of a free solvable group $S_{rd} (r,d \geq 2$). Then the image $\bar{H}$   in the abelianization  $A_r = S_{rd}/S_{rd}'$  (that is a free abelian group of rank $r$) is a direct factor of rank $2$, and  $H$ is a free  solvable group of rank $2$ and class $d.$
\end{lemma}
 \begin{proof}
By Lemma \ref{le:2.3} $\bar{H}$ is verbally closed in $A_r$. By Theorem \ref{th:1.3} $\bar{H}$ is a direct factor of $A_r.$ Obviously, $\bar{H}$ is non-trivial. We need to prove that $\bar{H}$ is not cyclic. By induction on $d$, we can assume that the image of $H$ in $S_{rd-1} = S_{rd}/S_{rd}^{(d-1)}$ is generated by the image $g'$ of $g$ (so it is cyclic) and that $f  \in  S_{rd}^{(d-1)}$.   Let $f (z_1,...,  z_r)$ be an expression of $f$ in the free generators $z_1, ..., z_r$ of $S_{rd}.$ The equation $f (x_1,..., x_r) = f$  is solvable in $S_{rd}$. Hence, it has a solution $h_1,..., h_r$ in $H.$ Then we can write the components in the form  $h_i = g^{t_i} f^{\alpha_i}, t_i \in \mathbb{Z }, \alpha_i \in \mathbb{Z}[g'], i = 1, ..., x_r.$  Then  $f (h_1,..., h_r) = f (g^{t_1}, ...,  g^{t_r})f^{\delta} =  f^{\delta},$ where $\delta = \sum_{i=1}^r\alpha_id_i(f)$ and $d_i(f)$ is a value of $i$-th partial Fox  derivation  in $\mathbb{Z}[g'].$  Details about Fox derivatives see in \cite{Romess}. It is easy to compute $\delta$ directly without Fox derivatives. We can collect exponents of $f$ from the expression  $f (h_1,..., h_r).$ Note that $\delta$ belongs to the fundamental ideal of $\mathbb{Z}[g']$, because $f (x_1,..., x_r)$ is a commutator word. Then  $f^{1-\delta} = 1$ and $1- \delta \not= 0.$  This is impossible because  $S_{rd}^{(d-1)}$ has no module torsion   (see \cite{Romess}).

By the G. Baumslag's result \cite{B} $H = gp(g, f)$ is free solvable group of rank $2$ and class $d$ with  basis $g, f.$ Recall that Baumslag proved in \cite{B} the following statement:  A subgroup H of a free solvable group $S$ is itself a free solvable group if and only if there exists a set $Y$ of generators of $H$ which freely generates, modulo some term of the derived series of $S$, a free abelian group. He also gave an obvious formula of the solvability class of $H.$ See also \cite{M}. 
\end{proof}

\section{Description of verbally and algebraically  closed cyclic subgroups of  free solvable groups}
\label{se:cyclic}

We start with the case of a cyclic subgroup. 

\begin{lemma}
\label{le:cycret}
{\it Let $S$ be a free solvable group of a finite rank $r$ and let  $H = gp(h)$  be a cyclic subgroup of $S$ generated by a non-trivial element $h \in S$. Then the following conditions are equivalent:
\begin{itemize}
\item [1)] $H$ is verbally closed in $S$;
\item [2)] $H$ is a retract of $S$;
\item [3)] the image of $h$ in the abelianization $A_r=S/S'$ (that is free abelian group  of rank $r$)  is primitive. 
\end{itemize}}
\end{lemma}

\begin{proof}
Let $\{z_1, ..., z_r\}$ be a basis of $S$.
The element $h$ can be expressed uniquely in the form
\begin{equation}
\label{eq:elfr} 
h = z_1^{k_1} ... z_r^{k_r}h'(z_1, \ldots , z_r), 
\end{equation}
where $ k_1, ..., k_r \in \mathbb{Z}$ and $ h'(z_1, \ldots, z_r)$ is a product of commutators of  words in $z_1, \ldots, z_r$ (a commutator word).  Then $h'\in S'.$ 

To show that 1) $\rightarrow$ 3) assume that $h$   has a non primitive image  in $ A_r$, i.e.,   either $h \in S'$   or $gcd(k_1, ..., k_n) = d > 1.$

We first suppose  that  $h \in S'$, so $ k_1 = \ldots =  k_r = 0$.  Replacing  each $z_i$ by  a  variable $x_i$ in (\ref{eq:elfr}) one gets an equation $h = x_1^{k_1} ... x_r^{k_r}h'(x_1, \ldots,x_r)$, with $h$ as a constant from $H$,  which has a solution in $S$.  However,  this equation does not have a solution in $H$, since $H$ is abelian, so $h'(h_1, \ldots,h_r) = 1$ for any $h_1, \ldots h_r \in H$. This shows that $H$ is not verbally closed in $S$ - contradiction. So $h \not\in S'$.  Then in this case  $gcd(k_1, ..., k_r) = d > 1.$ The equation 
 $$h = x_1^{k_1} ... x_r^{k_r}h'(x_1, \ldots,x_r)$$
  still has a solution in $S$, but for any $h_1, \ldots h_r \in H$ one has 
  $$h_1^{k_1} ... h_r^{k_r}h'(h_1, \ldots,h_r) = h_1^{k_1} ... h_r^{k_r}  = h^{ds} \neq h,$$
  for some $s \in \mathbb{Z}$. Hence, the equation does not have a solution in $H$, so $H$ is not verbally closed - contradiction.  

  To show that 3) $\rightarrow$ 2) assume that $h$
is  primitive in the abelianization of $S$. Then  there are integers $l_1, ..., l_r$ such that $k_1l_1 + ... k_rl_r = 1.$ Now we  define a homomorphism  $\varphi : S \rightarrow H = gp(h)$ by putting $\varphi (z_i) = h^{l_i}$ for $i = 1, ..., r.$ 
Since $H$ is abelian $\varphi(h') = 1$, so $\varphi (h) = h$ and $\varphi$ is a retraction. Hence $H$ is a retract, as claimed.

 2) $\rightarrow$ 1)   follows from Proposition \ref{pr:2.1}, statement 1). 
\end{proof}

\section{Description of verbally and algebraically  closed two-generated subgroups of  free metabelian groups}
\label{se:cyclic}

Further in the Section, $z_1, ..., z_r$ is denoted a basis of group $G$ that is $M_r$ or $M_{rk}.$ 

\begin{theorem}
\label{th:4.1}
Let $H = gp(g, f)$ be a two-generated subgroup of $G$ that is $M_r$ or $M_{rk}$, $r\geq 2, k \geq 4.$ Then  the following conditions are equivalent:
\begin{enumerate}
\item [1)] $H$ is a retract of $G$.
\item [2)] $H$ is a algebraically closed  subgroup of $Gr$.
\item [3)] $H$ is an verbally closed subgroup of $G$.
\end{enumerate}
\end{theorem}

\begin{proof} The implications  1) $\rightarrow$ 2)  $\rightarrow$ 3) are obvious. We are to prove that 3) $\rightarrow$ 1). By Lemma \ref{le:2.3} the image of $H$ in the abelianization 
$A_r = G/G'$ is verbally closed. By Lemma \ref{le:2.4} this image is a direct factor of $A_r$ of rank 2. We can assume that $g = uz_1, u ϵ G'$ and $f = vz_2, v ϵ G'$.   Let 
$ g (z_1,..., z_r)$ and  $f (z_1, ..., z_r)$ be expressions of $g$  and  $f$ respectively.  Then an equation 
       
\begin{equation}
\label{eq:3}
(f (x_1,..., x_r), g (x_1,..., x_r), g (x_1,..., x_r), f (x_1,..., x_r)) = (f, g, g, f)
\end{equation}                       
\noindent
has a solution in $G.$ Thus it has a solution in $H.$ It means that there are elements  $h_1,..., h_r $ in $ H$ such that

\begin{equation}
\label{eq:4}
(f (h_1,..., h_r), g (h_1,..., ,h_r), g (h_1,..., ,h_r), f (h, ..., h_r)) = (f, g, g, f).                             
\end{equation}

Now we consider the case $G = M_r.$ 
By the Baumslag's theorem (see above)  (\ref{eq:4}) is valid in the free metabelian group $H \simeq M_2$. By Lemma \ref{le:2.2}
 \begin{equation}
\label{eq:5}
f (h_1,..., h_r) \equiv f^{\epsilon}(mod\ H'),  g (h_1,..., h_r) \equiv g^{\eta}(mod\  H')\  {\textrm where}\  \epsilon , \eta \in \{\pm 1\}. 
\end{equation}

Recall that an element $u$ of a group $G$ is said to be {\em test element} if any endomorphism $\alpha : G \rightarrow G$  for which  $\alpha (g) = g$ is an automorphism of $G.$ Timoshenko proved in \cite{Timmet} that every non-trivial element of $M_2'$ is a test element for $M_2.$ Hence the following map 
\begin{equation}
\label{eq:6}                           
g \mapsto g (h_1,..., h_r), f  \mapsto  f (h_1,..., h_r)                                                        
\end{equation}
\noindent
defines an automorphism of $H.$ Let $\epsilon , \eta = 1.$ Then this automorphism is identical  $(mod\ H').$ By Bachmuth's theorem \cite{Bachmuth} it is an inner automorphism (see also \cite{Romnormal}).  Then there is an element $v \in  H$ for which  
\begin{equation}
\label{eq:7}
                            g (h_1,..., h_r) = g^v,   f (h_1,..., h_r) = f^v.                                                    
\end{equation}
It follows that $(f, g, g, f)^v= (f, g, g, f).$  Then $v \in  H'$ because the centralizer of any non-trivial element of the commutant $M_r'$ of $M_r$ is equal to $M_r'$. This statement was proved by Mal'cev in  \cite{M}. 
  
Define an endomorphism of $M_r$ by the map
\begin{equation}
\label{eq:8}
\beta: z_i  \mapsto h_i^{v^{-1}}, i =1, ..., r.
\end{equation} 
The image $\beta (M_r)$ lies in  $H.$  Moreover, 
$$\beta (g) = \beta (g (z_1,..., z_r)) = g (h_1,... ,h_r)^{v^{-1}} = g,$$
\begin{equation}
\label{eq:9}  
\beta (f) = \beta(f (z_1,..., z_r)) = f (h_1,..., h_r)^{v^{-1}} = f.
\end{equation}                  
Thus $\beta$ is identical on $H$ and so $\beta $ is a retraction and $H$ is a retract. 

Now let  $(\epsilon , \eta) \not= (1, 1).$ As above the following map 
\begin{equation}
\label{eq:10}
 g^{\eta} \mapsto g (h_1,..., h_r), f ^{\epsilon} \mapsto  f (h_1,..., h_r)                                                        
\end{equation}
\noindent
defines an automorphism of $H.$  A composition of this automorphism with itself is identical $(mod \ H')$ and we can finish our proof as above.

Let $G = M_{rk}.$ This group is free in a nilpotent variety ${\cal N}$. Then a pair of elements  $g, f \in H$ which induce a pair of free generators $\bar{g}, \bar{f}$ in $A_r=M_{rk}/M_{rk}'$, generate a free ${\cal N}$-factor in $M_{rk}$. This factor is obviously a retract of $M_{rk}.$  
\end{proof}

\section{Description of verbally and algebraically  closed two-generated subgroups of  free solvable groups}
\label{se:5}

To prove the main result of this section in the case of group $S_{r3}$ we will use a concrete example of a test element  of    $S_{23}$, that was constructed in \cite{Romtest}. The proof in the general case of group $S_{rd}$ will be slightly different.

\begin{example}
\label{ex:5.1}
 Let $x, y$ be a basis of $S_{23}.$  For any pair of positive integers $k$ and $l$  we denote by  $z(k; l)(x, y)$ the element  $(y, x; k, y; l-1)$, and by      $w(k; l; m; n)(x, y)$ we denote the element $(z(k; l)(x,y), z(m; n)(x,y)). $  Then
\begin{equation}
\label{eq:11}
u(x, y)  = w(3; 2; 1; 1)(x, y) w(2; 2; 1; 2)(x, y)                                               
\end{equation}
\noindent
is a test element of $S_{23}$.
\end{example}

The following lemma was proved in \cite{Romtest}. 

\begin{lemma}
\label{le:5.2}
Let $\phi $ be an endomorphism of $S_{23}$ for which $\phi (u) \equiv u (mod\ \gamma_8S_{23})$ where $u$ is the element in Example \ref{ex:5.1}. Then $\phi $ is an automorphism identical modulo $S_{23}'.$  In this case $\phi $ is inner by  theorem proved in \cite{BFM} (see also \cite{Romnormal}). 
\end{lemma}

The following more general result was proved in \cite{Timsolv}

\begin{lemma}
\label{le:5.3}
Let $S_{2d}$ be a free solvable group of rank $2$ and class $d \geq 2$ with basis $z_1, z_2,$  and let $v \in S_{2d}^{(d-1)}, v \not= 1.$ Then there is a positive number $m$ such that element $u = v^{(1-z_1^m)(1-z_2^m)}$ is a test element. Moreover, for every automorphism $\phi$ if $\phi (u) = u$ then  $\varphi = \phi^2 $ is an inner automorphism.
\end{lemma}

\begin{theorem}
\label{th:5.4}
Let $H = gp(g, f)$ be a two-generated subgroup of  $S_{rd}, r, d \geq 2.$ Then the following conditions are equivalent:
\begin{enumerate}
\item [1)] $H$ is a retract of $ S_{rd}$.
\item [2)] $H$ is a algebraically closed  subgroup of $ S_{rd}$.
\item [3)] $H$ is an verbally closed subgroup of $S_{rd}$.
\end{enumerate}
\end{theorem}
\begin{proof} The implications  1) $\rightarrow$ 2)  $\rightarrow$ 3) are obvious. We are to prove that 3) $\rightarrow$ 1). 

Case of group $S_{r3}.$ Let $z_1, ..., z_r$ be a basis of $S_{r3}$. By Theorem \ref{th:4.1} and Lemmas \ref{le:2.2} and \ref{le:2.3} we can assume that $g = vz_1, v \in S_{r3}',$ and   $f = wz_2, w \in  S_{r3}'$.    Let $g (z_1, ..., z_r)$ and   $f (z_1, ..., z_r)$ be two expressions of $g$ and $f$ respectively. Let $u$ be the element described by Example \ref{ex:5.1}. Then equation
\begin{equation}
\label{eq:12}
u(g(x_1, ..., x_r), f (x_1, ..., x_r)) = u(g, f)
\end{equation} 
\noindent
has a solution in $S_{r3}$. Hence it has a solution $h_1,... , h_r$ in $H. $ The map $ z_i  \mapsto h_i, i = 1, ..., r,$ defines an  endomorphism $S_{r3}$ with the image in $H.$ It maps $g = g(z_1, ..., z_r)$ to $g(h_1, ..., h_r),$ and  $f = f(z_1,  ..., z_r) $ to $f(h_1, ..., h_r).$
By Lemma \ref{le:5.2} the map $g \mapsto g(h_1, ..., h_r), f \mapsto  f(h_1,..., h_r)$ defines an inner automorphism of  $H.$   Then there is $ t \in  H$ for which 
\begin{equation}
\label{eq:13}
g (h_1, ..., h_r) = g^t,   f (h_1, ..., h_r) = f^t.                                                                
\end{equation}

It follows that $u(g,f)^t= u(g, f).$  In any group $S_{rd}, r, d \geq 2, $ the centralizer of any non-trivial element $y \in S_{rd}^{(d-1)}$ coincides with $S_{rd}^{(d-1)}$ \cite{M}. Hence  $t \in  H^{(2)}.$ Let $\phi$ be an endomorphism of $S_{r3}$ defined by the map: $z_i  \mapsto h_i^{t^{-1}}, i = 1, ..., r.$ The image $\phi (S_{r3})$ lies in  $H$, and 
\begin{equation}
\label{eq:14}
\phi (g) = \phi (g (z_1,…,z_r)) = g(h_1,…,h_r) ^{t^{-1}} = g, \phi (f) = \phi (f (z_1,…,z_r))  
(f (h_1,…,h_r))^{t^{-1}} = f.
\end{equation}
Hence $\phi$ is a retraction, and $H$ is a retract.

Case of group $S_{rd}, d \geq 2$. Now $u$ is defined in Lemma \ref{le:5.3}. We repeat all the arguments as above until we get the map $g \mapsto g(h_1, ..., h_r), f \mapsto  f(h_1,..., h_r)$ which defines an automorphism $\phi$ of $H.$  Then $\phi (S_{rd}) \subseteq H.$  By Lemma \ref{le:5.3} $\phi^2$ is an inner automorphism of  $H.$ For some $t\in H$ we have equalities (\ref{eq:13}). Then we finish proof as in the previous case. 

\end{proof}

\section{Open problems}
\label{se:problems}
\begin{problem}
\label{pr:6.1}
Is it true that for any subgroup $H$ of any free solvable group $S_{rd}$ of rank $r \geq 2$ and class $d \geq 2$ the following conditions are equivalent: 
\begin{enumerate}
\item [1)] $H$ is a retract of $ S_{rd}$,
\item [2)] $H$ is a algebraically closed  subgroup of $ S_{rd}$,
\item [3)] $H$ is an verbally closed subgroup of $S_{rd}$?
\end{enumerate}
\end{problem}
It is likely that the answer to this question is negative. By Timoshenko's theorem (see \cite{Timsolv} or \cite{Timbook}) the test rank of $S_{rd}, r, d \geq 2,$ is $r - 1.$ Recall, that {\em test rank} of a group $G$ is the minimal number of elements of $G$ such that every endomorphism fixing any of this elements is automorphism.  In the general case $H$ does not contain a test element. Hence, the methods of this paper do not work completely.

\begin{problem}
\label{pr:6.2}
Is it true that for any two-generated subgroup $H$ of any free polynilpotent group $P$ of rank $r \geq 2$ and class $(1, c_1, ..., c_l), l \geq 1,$ the following conditions are equivalent: 
\begin{enumerate}
\item [1)] $H$ is a retract of $ P$,
\item [2)] $H$ is a algebraically closed  subgroup of $ P$,
\item [3)] $H$ is an verbally closed subgroup of $P$?
\end{enumerate}
\end{problem}
By Timoshenko's theorem \cite{Timsolv} (see also \cite{Timbook}) any free solvable group $S_{2d}, d \geq 2,$ contains test elements, all belong to $S_{rd}^{(d-1)}.$ For the affimative solution  to this problem  sufficient to prove that  any $S_{2d}$ contains a test element  with properties as in Lemma \ref{le:5.2}. Note that the reference to the statement in \cite{BFM} (or \cite{Romnormal}) in Lemma \ref{le:5.2} 
works in the general case. Namely, any automorphism of   $S_{2d}, d \geq 2$ identical $mod\ S_{2d}', d \geq 2,$ is inner. 
 
Note, that by Gupta and the second author's  theorem \cite{Timpol} (see also \cite{Timbook}) any free polynilpotent group of class $(1, c_1, ..., c_l)$ for $ l \geq 1$ has a test element. We need in analog of Lemma \ref{le:2.4} to solve this problem affirnatively.


\begin{thebibliography}{\hspace{0.5in}}


\bibitem{MR} A. Myasnikov, V. Roman'kov. Verbally closed subgroups of free groups. J. Group Theory. 2014. 17, 29-40.

\bibitem{BMRom}  G. Baumslag,  A. Myasnikov and  V.  Roman'kov.  Two theorems about
equationally Noetherian groups. J. Algebra. 1997.  194, 654-664.

\bibitem{BMR1} G. Baumslag, A. Myasnikov and V. Remeslennikov. Algebraic geometry
over groups I. Algebraic sets and ideal theory. J. Algebra. 1999.  219, 16-79.

\bibitem{Ore}  O. Ore.  Some remarks on commutators.  Proc. Amer. Math. Soc. 1951.   2, 307-314.

\bibitem{LBST} M. Liebeck, E. O'Brian,   A.  Shalev A.  and  P.  Tiep.  The Ore conjecture.  J. European Math. Soc. 2010.  12,  939-1008.

\bibitem{Romeq} V. Roman'kov. Equations over groups. Groups, Complexity, Cryptology. 2012. 4, No. 2, 191-239. 

\bibitem{Segal}  D. Segal.  Words: notes on verbal width in groups. London Math. Soc. Lect. Notes Ser.  361. Cambridge: Cambridge Univ. Press, 2009. 134 p.

\bibitem{Romess} V.A. Roman’kov. Essays in algebra and cryptology. Solvable groups. Omsk: Omsk State University Publishing House, 2017. 207 p.

\bibitem{RH} V.A. Roman'kov, N.G. Khisamiev. Verbally and existentially closed subgroups of free nilpotent groups. Algebra and Logic. 2013. 52, No. 4, 336-351.

\bibitem{Romeqmet} V.A. Roman'kov. Equations in free metabelian groups. Siberian Mathematical Journal. 1979. 20, No. 3, 469-471. 

\bibitem{B} G. Baumslag.  Some subgroup theorems for free ${\upsilon}$-groups. Trans. Amer. Math. Soc. 1963.  108. 516–525.

\bibitem{M} A. I. Malcev.  On free soluble groups. Soviet Math. Dokl. 1960. 1, 65-68.


\bibitem{Timmet} E.I. Timoshenko. Test elements and test rank of a free metabelian group.  
Siberian Mathematical Journal. 2000.  41. No. 6. 1200-1204.

\bibitem{Bachmuth} S. Bachmuth. Automorphisms of free metabelian groups. Trans. Amer. Math. Soc. 1965. 118, 93-104.

\bibitem{Romtest} V.A. Roman'kov. Test elements for free solvable groups of rank 2. Algebra and Logic. 2001. 40, No. 2, 106-111. 

\bibitem{BFM} S. Bachmuth, E. Formanek, H.Y. Mochizuki. IA-automorphisms of two-generated torsion free groups.  J. Algebra. 1976.  40, 19-30. 

\bibitem{Romnormal} V.A. Roman'kov. Normal automorphisms of discrete groups. Siberian Mathematical Journal. 1983. 24, No. 4, 604-614.

\bibitem{Timsolv} E.I. Timoshenko. Computing test rank for a free solvable group. Algebra and Logic. 2006. 45, No. 4, 254-260.

\bibitem{Timbook} E.I. Timoshenko. Endomorphisms and universal theories of solvable groups (in Russian). Novosibirsk: NSTU Publishers. 2011. 327 p. ("NSTU Monograph" Series). 

\bibitem{Timpol} C.K. Gupta, E.I. Timoshenko. Test rank for some free polynilpotent groups. Algebra and logic. 2003. 42, 20-27. 
\end{thebibliography}
\end{document}